\documentclass[11pt]{amsart}

\usepackage{booktabs,amsmath}
\usepackage{amssymb}
\usepackage{amsthm}
\usepackage{graphics}
\usepackage{tikz}
\usetikzlibrary{arrows,automata,positioning}
\usepackage{tikz-qtree}
\usepackage{marvosym}
\usetikzlibrary{decorations.pathreplacing}
\usepackage{amsfonts}
\usepackage{mathtools}
\usepackage{algorithm,algpseudocode}
\usepackage{mathdots}
\usepackage{float}
\usepackage{caption}
\usepackage{subcaption}
\usepackage{stmaryrd}
\usepackage{mathrsfs}
\usepackage{soul}
\allowdisplaybreaks

\usepackage{color} 
\definecolor{ao(english)}{rgb}{0.0, 0.5, 0.0}

\newtheorem{lemma}{Lemma}
\newtheorem{theorem}{Theorem}
\newtheorem{proposition}{Proposition}
\newtheorem{corollary}{Corollary}

\begin{document}
\title{Subgroups of $SL_2(\mathbb{Z})$ characterized by certain continued fraction representations}
\author{Sandie Han, Ariane M. Masuda, Satyanand Singh, and Johann Thiel}
\date{\today}
\address{Department of Mathematics, New York City College of Technology, The City University of New York (CUNY), 300 Jay Street,
Brooklyn, New York 11201}
\email{\{shan,amasuda,ssingh,jthiel\}@citytech.cuny.edu}

\begin{abstract}
For positive integers $u$ and $v$, let $L_u=\begin{bmatrix} 1 & 0 \\ u & 1 \end{bmatrix}$ and $R_v=\begin{bmatrix} 1 & v \\ 0 & 1 \end{bmatrix}$. Let $S_{u,v}$ be the monoid generated by $L_u$ and $R_v$, and $G_{u,v}$ be the group generated by $L_u$ and $R_v$. In this paper we expand on a characterization of matrices $M=\begin{bmatrix}a & b \\c & d\end{bmatrix}$ in $S_{k,k}$ and $G_{k,k}$ when $k\geq 2$ given by Esbelin and Gutan to $S_{u,v}$ when $u,v\geq 2$ and $G_{u,v}$ when $u,v\geq 3$. We give a simple algorithmic way of determining if $M$ is in $G_{u,v}$ using a recursive function and the short continued fraction representation of $b/d$. 
\end{abstract}
\maketitle

\section{Introduction}
For positive integers $u$ and $v$, let $L_u=\begin{bmatrix} 1 & 0 \\ u & 1 \end{bmatrix}$ and $R_v=\begin{bmatrix} 1 & v \\ 0 & 1 \end{bmatrix}$. Let $S_{u,v}$ be the monoid generated by $L_u$ and $R_v$, and $G_{u,v}$ be the group generated by $L_u$ and $R_v$. In this case, the membership problem refers to the issue of  determining if there are relatively simple descriptions of the matrices in $S_{u,v}$ and $G_{u,v}$. We begin with some definitions and current results on this problem.

Using similar notation to that in~\cite{EG1,EG2}, let 
$$\mathscr{S}_{u,v} = \left\{\begin{bmatrix}1+uvn_1 & vn_2\\ un_3 & 1+uvn_4\end{bmatrix}\in SL_2(\mathbb{N})\enskip\colon\enskip(n_1,n_2,n_3,n_4)\in\mathbb{N}^4\right\}$$ and $$\mathscr{G}_{u,v} = \left\{\begin{bmatrix}1+uvn_1 & vn_2\\ un_3 & 1+uvn_4\end{bmatrix}\in SL_2(\mathbb{Z})\enskip\colon\enskip(n_1,n_2,n_3,n_4)\in\mathbb{Z}^4\right\}.$$

\begin{proposition}\label{GuvIN}
For any integers $u$ and $v$, we have that $\mathscr{S}_{u,v}$ is a monoid and $\mathscr{G}_{u,v}$ is a group. Furthermore, for any integers $u,v\geq 2$, we have that $S_{u,v}\subseteq\mathscr{S}_{u,v}$ and $G_{u,v}\subseteq\mathscr{G}_{u,v}$.
\end{proposition}

The proof that $\mathscr{S}_{u,v}$ is a monoid and $\mathscr{G}_{u,v}$ is a group follows readily from two computations:

\begin{align*}
& \begin{bmatrix}1+uvn_1 & vn_2\\ un_3 & 1+uvn_4\end{bmatrix}\cdot\begin{bmatrix}1+uvm_1 & vm_2\\ um_3 & 1+uvm_4\end{bmatrix} \\ =& \begin{bmatrix}1+uv(n_1+m_1+uvn_1m_1+n_2m_3) & v((1+uvn_1)m_2+(1+uvm_4)n_2)\\ u((1+uvm_1)n_3+(1+uvn_4)m_3) & 1+uv(n_4+m_4+uvn_4m_4+n_3m_2)\end{bmatrix}
\end{align*}

and

\begin{align*}
& \begin{bmatrix}1+uvn_1 & vn_2\\ un_3 & 1+uvn_4\end{bmatrix}^{-1} = \begin{bmatrix}1+uvn_4 & -vn_2\\ -un_3 & 1+uvn_1\end{bmatrix}.
\end{align*}

Note that, since $S_{u,v}$ and $G_{u,v}$ are generated freely~\cite{N1,LS} for $u,v\geq 2$, every matrix in either can be written as an alternating product of non-zero powers of $L_u$ and $R_v$ in a unique way. From this fact, the rest of the proof of Proposition~\ref{GuvIN} follows by induction (see~\cite[Proposition 1, Theorem 4]{CGS}  for the $u=v$ version).

The following classical result due to Sanov~\cite{S} shows that $\mathscr{G}_{2,2}\subseteq G_{2,2}$. As Theorem~\ref{cgs} below shows, this does not hold true in general.

\begin{theorem}[Sanov~\cite{S}]\label{sanov}
We have that $G_{2,2}=\mathscr{G}_{2,2}$.
\end{theorem}

\begin{theorem}[Chorna, Geller, and Shpilrain~\cite{CGS}]\label{cgs}
The subgroup $G_{k,k}$ with $k\geq 3$
has infinite index in the group $\mathscr{G}_{k,k}$.
\end{theorem}

Given a rational number $q$, if there exist integers $q_0,q_1,\dots,q_r$ (referred to as partial quotients) such that $$q = q_0+\cfrac{1}{q_1+\cfrac{1}{q_2+\ddots+\cfrac{1}{q_r }}},$$ then we refer to such an identity as a continued fraction representation of $q$ and denote it by $[q_0,q_1,\dots,q_r].$ Every rational number has multiple continued fraction representations, however each has a unique representation that satisfies the additional requirements that $q_i\geq 1$ when $0<i< r$, and $q_r>1$ when $r>0.$ We refer to this representation as \emph{the} short continued fraction representation of $q.$

With the above definitions, we can now state equivalent forms of some results of Esbelin and Gutan~\cite{EG2,EG1} on the membership problem for $S_{k,k}$ and $G_{k,k}$ for $k\geq 2.$

\begin{theorem}[Esbelin and Gutan~\cite{EG2}]\label{egMon}
Suppose that $M=\begin{bmatrix}a & b \\c & d\end{bmatrix}\in\mathscr{S}_{k,k}$ for some $k\geq 2$. Then $M\in S_{k,k}$ if and only if at least one of the rationals  $c/a$ and $b/d$ has a continued fraction expansion having all partial quotients in $k\mathbb{N}.$
\end{theorem}

\begin{theorem}[Esbelin and Gutan~\cite{EG1}]\label{egGp}
Suppose that $M=\begin{bmatrix}a & b \\c & d\end{bmatrix}\in\mathscr{G}_{k,k}$ for some $k\geq 2$. Then $M\in G_{k,k}$ if and only if at least one of the rationals $c/a$ and $b/d$ has a continued fraction expansion having all partial quotients in $k\mathbb{Z}.$
\end{theorem}

In this paper we generalize Theorems~\ref{egMon} and~\ref{egGp} and give a simple proof of Theorem~\ref{sanov}. In particular, we give a closely related characterization of matrices in $G_{u,v}$, when $u,v\geq 3$, in terms of short continued fraction representations. It follows from Theorem~\ref{egGp} that membership in $G_{k,k}$, when $k\geq 2$, is equivalent to determining if certain rational numbers have a continued fraction representation with partial quotients satisfying specific criteria. Esbelin and Gutan developed an algorithm in~\cite{EG1} designed to search for the desired continued fraction directly using the Division Algorithm. We show that we can begin with the short continued fraction representation (which is straightforward to compute and also uses the Division Algorithm) of {\emph{one}} of these rational numbers (by Proposition~\ref{oneFracGroup}) and apply a simple function whose output completely determines whether or not the continued fraction representations needed in Theorem~\ref{egGp} exist. That is, the desired continued fraction representation in Theorem~\ref{egGp} (if it exists) is still encoded in the short continued fraction representation, albeit in a way that is not immediately obvious. Our motivation for using short continued fractions to study the members of $G_{u,v}$ comes from the success of this approach when applied to the members of $S_{u,v}$. 

The paper is organized as follows. In Section 2 we go over some notation, definitions, and results that will allow us to recast the membership problem in terms of the properties of certain sets of vectors. Section 3 contains our main results extending Theorems~\ref{egMon} and~\ref{egGp} in the following two ways. We generalize Esbelin and Gutan's results for the case $u=v$, and we remove the need to state our results in terms of properties of $c/a$ or $b/d$ in favor of just one. In the last section, Section 4, we give an example that illustrates how we solve the membership problem in the group case based on Theorem~\ref{sanovlikeGp}.



\section{Some necessary definitions and results}

In this section we will define several operations and functions on a set of vectors. These vectors will act as stand-ins for continued fraction representations. This is done to avoid any issues regarding well-definedness.

Let $A=\bigcup_{r=0}^\infty (\mathbb{Z}\times\mathbb{Z}_{\neq 0}^r)$. We denote an element of $A$ by $\llbracket q_0,q_1,q_2,\dots,q_r\rrbracket.$ Let $$-\llbracket q_0,q_1,q_2,\dots,q_r\rrbracket :=\llbracket -q_0,-q_1,-q_2,\dots,-q_r\rrbracket.$$ For any nonnegative integers $m$ and $n$, let $$\llbracket  q_0,q_1,q_2,\dots,q_m\rrbracket\oplus\llbracket  p_0,p_1,p_2\dots,p_n\rrbracket := \begin{cases}
\llbracket q_0,q_1,q_2,\dots,q_m,p_0,p_1,p_2\dots,p_n\rrbracket & \text{ if }p_0\neq 0,\\
\llbracket q_0,q_1,q_2,\dots,q_m+p_1,p_2\dots,p_n\rrbracket &\text{ otherwise.}
\end{cases}$$ 
Let 
\begin{align*}
    A_0 &= \{\llbracket q_0,q_1,q_2,\dots,q_r\rrbracket\in A: [q_i,\dots,q_r]\neq 0\text{ when } 0<i<r\},\\
    A_1 &=\{\llbracket q_0,q_1,q_2,\dots,q_r\rrbracket\in A_0:q_i\geq 1\text{ when } 0<i<r,\text{ and } q_r>1\text{ when } r>0\}, \text{ and}\\
    A_2 &=\{\llbracket q_0,q_1,q_2,\dots,q_r\rrbracket\in A_0:|q_i|> 1 \text{ when } 0<i\leq r\}.
\end{align*} 
Define the function $C:\mathbb{Q}\to A_1$ by 
$$C(x)=\llbracket x_0,x_1,x_2,\dots,x_r\rrbracket$$ if $[x_0,x_1,x_2,\dots,x_r]$ is the short continued fraction representation of $x$. Note that by the uniqueness of short continued fraction representations, $C(\mathbb{Q})=A_1$. Let $E:A_0\to\mathbb{Q}$ be given by $$E(\llbracket q_0,q_1,q_2,\dots,q_r\rrbracket) = q_0+\cfrac{1}{q_1+\cfrac{1}{q_2+\ddots+\cfrac{1}{q_r }}}.$$ Note that while $(E\circ C)(x)=x$, $E$ is not a one-to-one function, so these two functions are not inverses of each other.

\begin{lemma}\label{tran}
Suppose that $\llbracket q_0,q_1,q_2,\dots,q_r\rrbracket\in A_1$ is such that $q_j=1$ for some $0<j<r$. Then
\begin{align*}
    E(\llbracket q_0,q_1,q_2,\dots,q_{j-1},1,q_{j+1},\dots,q_r\rrbracket) &= E(\llbracket q_0,q_1,q_2,\dots,q_{j-1}+1\rrbracket\oplus-\llbracket q_{j+1}+1,\dots,q_r\rrbracket).
\end{align*}
\end{lemma}

\begin{proof}
The proof of the lemma makes use of the following identity:

\begin{align}\label{cfid}
\alpha+\frac{1}{1+\cfrac{1}{\beta}} = \alpha+1-\cfrac{1}{\beta+1}
\end{align}
for any $\alpha,\beta\in\mathbb{R}$ with $\beta\neq -1, 0.$ More specifically,

\begin{align*}
E(\llbracket q_0,q_1,q_2,\dots,q_{j-1},1,q_{j+1},\dots,q_r\rrbracket) &= q_0+\cfrac{1}{q_1+\cfrac{1}{q_2+\ddots+\cfrac{1}{q_{j-1}+\cfrac{1}{1+\cfrac{1}{q_{j+1}+\ddots+\cfrac{1}{q_r}}}}}}\\
&= q_0+\cfrac{1}{q_1+\cfrac{1}{q_2+\ddots+\cfrac{1}{q_{j-1}+1-\cfrac{1}{q_{j+1}+1+\ddots+\cfrac{1}{q_r}}}}}\\
&= E(\llbracket q_0,q_1,q_2,\dots,q_{j-1}+1,-(q_{j+1}+1),\dots,-q_r\rrbracket)\\
&= E(\llbracket q_0,q_1,q_2,\dots,q_{j-1}+1\rrbracket\oplus-\llbracket q_{j+1}+1,\dots,q_r\rrbracket).
\end{align*}

\end{proof}

\begin{lemma}\label{tran2}
Suppose that $\llbracket q_0,q_1,q_2,\dots,q_r\rrbracket\in A_2$ is such that $q_j<0$ for some $0<j<r$. Then
\begin{align*}
    E(\llbracket q_0,q_1,q_2,\dots,q_r\rrbracket) &= E(\llbracket q_0,q_1,q_2,\dots,q_{j-1}-1,1\rrbracket\oplus-\llbracket q_j+1,\dots,q_r\rrbracket).
\end{align*}
\end{lemma}

\begin{proof}
The proof of the lemma follows from using~\eqref{cfid} from right to left similarly to the proof of Lemma~\ref{tran}.
\end{proof}

Lemmas~\ref{tran} and~\ref{tran2} show that we can manipulate elements in $A_1$ and $A_2$ in ways that do not disrupt the rational number that they represent. These manipulations are critically dependent on~\eqref{cfid} and are the key to our main results. Since Lemmas~\ref{tran} and~\ref{tran2} can be applied multiple times, we are naturally led to the following function definitions.

Define a function $f:A_1\to A_2$ recursively by
\begin{center}
\begin{align*}
&f(\llbracket q_0,q_1,q_2,\dots,q_r\rrbracket) \\
=&\begin{cases}
\llbracket q_0,q_1,q_2,\dots,q_r\rrbracket  &\text{ if $r=0$ or $q_i\neq 1$ for $0<i< r$,}\\
\llbracket q_0,q_1,q_2,\dots,q_{j-1}+1\rrbracket\oplus-f(\llbracket q_{j+1}+1,q_{j+2},\dots,q_r\rrbracket) &\text{ if $q_j=1$ and $q_i\neq 1$ for $0< i < j$,}
\end{cases}
\end{align*}
\end{center}

\noindent and a function $g:A_2\to A_1$ recursively by
\begin{center}
\begin{eqnarray*}
&&g(\llbracket q_0,q_1,q_2,\dots,q_r\rrbracket)\\
&=&\begin{cases}
\llbracket q_0,q_1,q_2,\dots,q_r\rrbracket &\text{ if $r=0$ or $q_i>0$ for $0<i\leq r$,}\\
\llbracket q_0,q_1,q_2,\dots,q_{j-1}-1,1\rrbracket\oplus g(-\llbracket q_j+1,q_{j+1},\dots,q_r\rrbracket)& \text{ if $q_j<0$ and $q_i>0$ for $0<i<j$.}
\end{cases}
\end{eqnarray*}
\end{center}

The next lemma is straightforward and will be used to show how $E$ and $f$ (as well as $E$ and $g$) interact.


\begin{lemma}\label{Eid}
For $a,b\in A_0$ with $a=\llbracket q_0,q_1,q_2,\dots,q_r\rrbracket$ and $b=\llbracket p_0,p_1,p_2,\dots,p_s\rrbracket$ , 
$$E(a\oplus b) = 
\begin{cases}
q_0+\cfrac{1}{q_1+\cfrac{1}{q_2+\ddots+\cfrac{1}{q_r+\cfrac{1}{E(b)}}}} &\text{ if } p_0\neq 0,\\
q_0+\cfrac{1}{q_1+\cfrac{1}{q_2+\ddots+\cfrac{1}{q_r+E(b)}}} &\text{ if } p_0=0.
\end{cases}
$$
\end{lemma}

The corollaries below follow  from Lemmas~\ref{tran},~\ref{tran2}, and~\ref{Eid}.

\begin{corollary}\label{ef}
For all $\llbracket q_0,q_1,q_2,\dots,q_r\rrbracket\in A_1$, $E(\llbracket q_0,q_1,q_2,\dots,q_r\rrbracket) = (E\circ f)(\llbracket q_0,q_1,q_2,\dots,q_r\rrbracket).$
\end{corollary}

\begin{proof}
The proof follows by strong induction on $r$.
Suppose that $r=0$, then $(E\circ f)(\llbracket q_0\rrbracket) = E(\llbracket q_0\rrbracket).$
Now suppose that the corollary holds for $0\leq r\leq t$ for some $t\geq 0$. Let $\llbracket q_0,q_1,q_2,\dots,q_t,q_{t+1}\rrbracket\in A_1$. If $q_i\neq 1$ for all $0<i\leq t$, then the statement follows immediately from the definition of $f$. Otherwise, let $0<j\leq t$ be the smallest value such that $q_j = 1$. Then

\begin{eqnarray*}
    &&(E\circ f)(\llbracket q_0,q_1,q_2,\dots,q_{j-1},1,q_{j+1},\dots, q_t,q_{t+1}\rrbracket)\\ 
    &=& E(\llbracket q_0,q_1,q_2,\dots,q_{j-1}+1\rrbracket\oplus-f(\llbracket q_{j+1}+1,q_{j+2},\dots,q_{t+1}\rrbracket))\text{ by the definition of $f$}\\
    &=& q_0+\cfrac{1}{q_1+\cfrac{1}{q_2+\ddots+\cfrac{1}{q_{j-1}+1+\cfrac{1}{E(-f(\llbracket q_{j+1}+1,q_{j+2},\dots,q_{t+1}\rrbracket))}}}}\text{ by Lemma~\ref{Eid}}\\
    &=& q_0+\cfrac{1}{q_1+\cfrac{1}{q_2+\ddots+\cfrac{1}{q_{j-1}+1-\cfrac{1}{(E\circ f)(\llbracket q_{j+1}+1,q_{j+2},\dots,q_{t+1}\rrbracket)}}}}\\
    &=& q_0+\cfrac{1}{q_1+\cfrac{1}{q_2+\ddots+\cfrac{1}{q_{j-1}+1-\cfrac{1}{E(\llbracket q_{j+1}+1,q_{j+2},\dots,q_{t+1}\rrbracket)}}}}\text{ by the inductive hypothesis}\\
    &=& q_0+\cfrac{1}{q_1+\cfrac{1}{q_2+\ddots+\cfrac{1}{q_{j-1}+1-\cfrac{1}{q_{j+1}+1+\ddots+\cfrac{1}{q_{t+1}}}}}}\text{ by the definition of $E$}\\
    &=& q_0+\cfrac{1}{q_1+\cfrac{1}{q_2+\ddots+\cfrac{1}{q_{j-1}+\cfrac{1}{1+\cfrac{1}{q_{j+1}+\ddots+\cfrac{1}{q_{t+1}}}}}}}\text{ by~\eqref{cfid}}\\
    &=& E(\llbracket q_0,q_1,q_2,\dots,q_{j-1},1,q_{j+1},\dots, q_t,q_{t+1}\rrbracket).
\end{eqnarray*}
\end{proof}

\begin{corollary}\label{efc}
For all $x\in\mathbb{Q}$, $(E\circ f\circ C)(x)=x$.
\end{corollary}

\begin{proof}
Apply the previous corollary to $C(x)$.
\end{proof}

\begin{corollary}\label{eg}
For all $\llbracket q_0,q_1,q_2,\dots,q_r\rrbracket\in A_2$, $E(\llbracket q_0,q_1,q_2,\dots,q_r\rrbracket) = (E\circ g)(\llbracket q_0,q_1,q_2,\dots,q_r\rrbracket).$
\end{corollary}

\begin{proof}
The corollary follows by strong induction on $r$ using Lemma~\ref{tran2} and the definition of the function $g$ similarly to the proof of Corollary~\ref{ef}.
\end{proof}

Note that given a representation of a rational number $$q=q_0+\cfrac{1}{q_1+\cfrac{1}{q_2+\ddots+\cfrac{1}{q_r }}}$$ with $q_0\in\mathbb{Z}$, $|q_i|>1$ for $0<i<r$ and $|q_r|\geq 3$ when $r>0$, we can, by uniqueness, compute the short continued fraction representation of $q$ using~\eqref{cfid}. In particular, using Corollary~\ref{eg}, we get the following result.

\begin{lemma}\label{scf}
If $\llbracket q_0,q_1,\dots,q_r\rrbracket\in A_2$ and $|q_r|\geq 3$ when $r>0$, then $$g(\llbracket q_0,q_1,\dots,q_r\rrbracket)=(C\circ E)(\llbracket q_0,q_1,\dots,q_r\rrbracket)).$$
\end{lemma}

The following lemma is needed for the proof of Proposition~\ref{fg} below.

\begin{lemma}\label{firstEntry}
For all $\llbracket q_0,q_1,q_2,\dots,q_r\rrbracket\in A_2$, if $$g(\llbracket q_0,q_1,q_2,\dots,q_r\rrbracket) = \llbracket q'_0,q'_1,\dots,q'_s\rrbracket,$$ then for any integer $n$, $$g(\llbracket q_0+n,q_1,q_2,\dots,q_r\rrbracket) = \llbracket q'_0+n,q'_1,\dots,q'_s\rrbracket.$$
\end{lemma}

\begin{proof}
The result follows from noting that the first entry of $g(\llbracket q_0,q_1,q_2,\dots,q_r\rrbracket)$ is either $q_0$ or $q_0-1$, which is entirely determined by $q_1$ and nothing else.
\end{proof}

The following proposition shows that under certain circumstances, the functions $f$ and $g$ `undo' each other.

\begin{proposition}\label{fg}
For all $\llbracket q_0,q_1,q_2,\dots,q_r\rrbracket\in A_2$ with $|q_i|\geq 3$ for $0<i\leq r$, $$(f\circ g)(\llbracket q_0,q_1,q_2,\dots,q_r\rrbracket) = \llbracket q_0,q_1,q_2,\dots,q_r\rrbracket.$$
\end{proposition}

\begin{proof}
The proof can be broken down into two cases.

\noindent{\bf Case 1:} The statement is clearly true if $r=0$ or $q_i>0$ for all $0<i\leq r.$

\noindent{\bf Case 2:} Suppose that $q_j<0$ and $q_i>0$ for $0<i<j$ for some $0<j\leq r$. Then
\begin{eqnarray*}
&&(f\circ g)(\llbracket q_0,q_1,q_2,\dots,q_r\rrbracket) \\
&=& f(\llbracket q_0,q_1,q_2,\dots,q_{j-1}-1,1\rrbracket\oplus g(-\llbracket q_j+1,q_{j+1},\dots,q_r\rrbracket))\\
&=& f(\llbracket q_0,q_1,q_2,\dots,q_{j-1}-1,1\rrbracket\oplus \llbracket q'_{j+1},\dots,q'_s\rrbracket)\text{ for some $s$}\\
&=& f(\llbracket q_0,q_1,q_2,\dots,q_{j-1}-1,1,q'_{j+1},\dots,q'_s\rrbracket)\\
&=& \llbracket q_0,q_1,q_2,\dots,q_{j-1}\rrbracket\oplus-f(\llbracket q'_{j+1}+1,\dots,q'_s\rrbracket)\text{ since $q_i\geq 3$ for $0<i<j$}\\
&=& \llbracket q_0,q_1,q_2,\dots,q_{j-1}\rrbracket\oplus-f(g(-\llbracket q_j,q_{j+1},\dots,q_r\rrbracket))\text{ by Lemma~\ref{firstEntry}}\\
&=& \llbracket q_0,q_1,q_2,\dots,q_{j-1}\rrbracket\oplus-(f\circ g)(-\llbracket q_j,q_{j+1},\dots,q_r\rrbracket).
\end{eqnarray*}
The desired result follows from repeatedly applying the identity from Case 2 until the problem is reduced to Case 1.
\end{proof}

Another main tool in the next section will be the following lemma, which is a straightforward generalization of~\cite[Lemma 5]{HMST2} and is similar to~\cite[Lemma 4]{EG1} and~\cite[Lemma 2.1]{EG2}.

\begin{lemma}\label{cf}
 Let $a/b$ be a rational number with $a/b=[q_0,q_1,q_2,\ldots,q_r]$, $\alpha\in\mathbb{Z}$, and $u$ and $v$ be nonnegative integers. It follows that
 \begin{enumerate}
    \item[(a)] $L_u^\alpha\begin{bmatrix}a\\b\end{bmatrix}=\begin{bmatrix}a\\au\alpha+b\end{bmatrix}$ and $R_v^\alpha\begin{bmatrix}a\\b\end{bmatrix}=\begin{bmatrix}a+bv\alpha\\b\end{bmatrix}$;
    \item[(b)] if $q_0=0$ and $u\alpha + q_1\neq 0$, then $\cfrac{a}{au\alpha+b}=[0,
u\alpha+ q_1,q_2,\dots,q_r]$;
    \item[(c)] if $q_0=0$ and $u\alpha + q_1=0$, then $\cfrac{a}{au\alpha+b}=[q_2,\dots,q_r]$;
    \item[(d)] if $q_0\neq0$, then $\cfrac{a}{au\alpha+b}=[0,
u\alpha,q_0,q_1,q_2,\dots,q_r]$; 
    \item[(e)]  $\cfrac{a+bv\alpha}{b}=[v\alpha+q_0,q_1,q_2,\dots,q_r]$.
 \end{enumerate}
\end{lemma}


\section{The Main Theorems}

We say that $\llbracket q_0,q_1,q_2,\dots,q_r\rrbracket\in A$ satisfies the $(u,v)$-divisibility property if $v|q_i$ when $i$ is even and $u|q_i$ when $i$ is odd. Before we state and prove our main results, Theorems~\ref{sanovlike} and ~\ref{sanovlikeGp}, we begin by showing that the entries of matrices in $\mathscr{S}_{u,v}$ or $\mathscr{G}_{u,v}$ generate rational numbers with continued fraction representations that share a nice relationship.


\begin{proposition}\label{oneFracMon}
For integers $u,v\geq 2$ and a matrix $M=\begin{bmatrix}
    a & b \\
    c & d
    \end{bmatrix}\in \mathscr{S}_{u,v}$, $C(c/a)$ satisfies the $(v,u)$-divisibility property if and only if $C(b/d)$ satisfies the $(u,v)$-divisibility property.
\end{proposition}

\begin{proof}
Note that, by assumption, $a,d\neq 0$.

$(\Rightarrow)$ The proof in this direction is similar to that of the reverse direction. Since all subsequent results are stated in terms of $b/d$, we omit this proof.

$(\Leftarrow)$  Suppose that $C(b/d)$ has the $(u,v)$-divisibility property. We first show that it must be the case that $C(b/d)=\llbracket v\alpha_0,u\alpha_1,\dots,v\alpha_{r-1}\rrbracket$ for some odd $r$.

Suppose that $C(b/d)=\llbracket v\alpha_0,u\alpha_1,\dots,u\alpha_r\rrbracket$ and consider the matrix $$N=R_v^{-\alpha_{r-1}}\cdots L_u^{-\alpha_1}R_v^{-\alpha_0}M.$$ By repeatedly applying Lemma~\ref{cf} parts (c) and (e), we get that $N=\begin{bmatrix}
    a' & b' \\
    c' & d'
    \end{bmatrix}$ with $b'/d'=1/(u\alpha_r),$ a contradiction since $N\in\mathscr{G}_{u,v}$ by Proposition~\ref{GuvIN}.

So we must have that $C(b/d)=\llbracket v\alpha_0,u\alpha_1,\dots,v\alpha_{r-1}\rrbracket$ and, using the same definition for $N$ as above, $N=\begin{bmatrix}
    a' & 0 \\
    c' & d'
    \end{bmatrix}$. Since $N\in\mathscr{G}_{u.v}$, we must have that $a'=1$, $d'=1$ and $c'=u\alpha$ for some $\alpha\in\mathbb{Z},$ i.e., $N=L_u^\alpha$. Note that we cannot have $a'=-1$ and $d'=-1$, as these values are not congruent to $1\pmod{uv}$ as is required by the definition of $\mathscr{G}_{u,v}$. In particular, $M=R_v^{\alpha_0}L_u^{\alpha_1}\cdots R_v^{\alpha_{r-1}}L_u^\alpha.$

Since $M\in\mathscr{S}_{u,v}$, we must have that $\alpha\geq 0$; otherwise, $a$ or $c$ would be negative. 
By repeatedly applying Lemma~\ref{cf} parts (d) and (e) to the following exhaustive list of cases, we get that $c/a=[u\beta_0,v\beta_1,\dots,u\beta_k]$ for some $k\geq 0$ where $\beta_0\geq 0$ and $\beta_i>0$
for $0<i\leq k$.

\bigskip
\noindent{\bf Case 1:} If $r = 1$ and $\alpha_0=0$, then $c/a=[u\alpha].$

\noindent{\bf Case 2:} If $r = 1$, $\alpha_0>0$, and $\alpha=0$, then $c/a=[0].$

\noindent{\bf Case 3:} If $r = 1$, $\alpha_0>0$, and $\alpha>0$, then $c/a=[0,v\alpha_0,u\alpha].$

\noindent{\bf Case 4:} If $r > 1$, $\alpha_0=0$, and $\alpha=0$, then $c/a=[u\alpha_1, v\alpha_2,\dots, v\alpha_{r-3}, u\alpha_{r-2}].$

\noindent{\bf Case 5:} If $r > 1$, $\alpha_0>0$, and $\alpha=0$, then $c/a=[0,v\alpha_0,u\alpha_1,\dots, v\alpha_{r-3}, u\alpha_{r-2}].$

\noindent{\bf Case 6:} If $r > 1$, $\alpha_0=0$, and $\alpha>0$, then $c/a=[u\alpha_1, v\alpha_2, \dots,v\alpha_{r-1},u\alpha].$ 

\noindent{\bf Case 7:} If $r > 1$, $\alpha_0>0$, and $\alpha>0$, then $c/a=[ 0,v\alpha_0,u\alpha_1,\dots,v\alpha_{r-1},u\alpha].$

\bigskip
The result now follows from the fact that $C(c/a)=\llbracket u\beta_0,v\beta_1,\dots,u\beta_k\rrbracket$ by the definition of $C.$
\end{proof}

\begin{proposition}\label{oneFracGroup}
For integers $u,v\geq 3$ and a matrix $M=\begin{bmatrix}
    a & b \\
    c & d
    \end{bmatrix}\in \mathscr{G}_{u,v}$, $(f\circ C)(c/a)$ satisfies the $(v,u)$-divisibility property if and only if $(f\circ C)(b/d)$ satisfies the $(u,v)$-divisibility property.
\end{proposition}

\begin{proof}
$(\Rightarrow)$ As in Proposition~\ref{oneFracMon}, we omit this portion of the proof using the same rationale.

$(\Leftarrow)$ Following the same steps as in the proof of Proposition~\ref{oneFracMon}, we get that $c/a=[u\beta_0,v\beta_1,\dots,u\beta_k]$ where $\llbracket u\beta_0,v\beta_1,\dots,u\beta_k\rrbracket\in A_2.$ Note that in this case, we are not forced (or required) to conclude that the $\alpha$ obtained in the proof is nonnegative. Now

\begin{align*}
    (f\circ C)(c/a) &= (f\circ C)(E(\llbracket u\beta_0,v\beta_1,\dots,u\beta_k\rrbracket))\\
    &= (f\circ C\circ E)(\llbracket u\beta_0,v\beta_1,\dots,u\beta_k\rrbracket)\\
    &= (f\circ g)(\llbracket u\beta_0,v\beta_1,\dots,u\beta_k\rrbracket)\text{ by Lemma~\ref{scf}}\\
    &= \llbracket u\beta_0,v\beta_1,\dots,u\beta_k\rrbracket\text{ by Proposition~\ref{fg}.}
\end{align*}
\end{proof}

Propositions~\ref{oneFracMon} and~\ref{oneFracGroup} allow us to state our main results in terms of $b/d$ alone.

\begin{theorem}\label{sanovlike}
For integers $u,v\geq 2$ and a matrix $M=\begin{bmatrix}
    a & b \\
    c & d
    \end{bmatrix}\in \mathscr{S}_{u,v}$, $M\in S_{u,v}$ if and only if $C(b/d)$ satisfies the $(u,v)$-divisibility property.
\end{theorem}

The proof of Theorem~\ref{sanovlike} is similar to that of Theorem~\ref{sanovlikeGp} below. We therefore omit the proof. 


\begin{theorem}\label{sanovlikeGp}
For integers $u,v\geq 3$ and a matrix $M=\begin{bmatrix}
    a & b \\
    c & d
    \end{bmatrix}\in \mathscr{G}_{u,v}$, $M\in G_{u,v}$ if and only if $(f\circ C)(b/d)$ satisfies the $(u,v)$-divisibility property.
\end{theorem}

\begin{proof}
$(\Rightarrow)$ Suppose that $M\in G_{u,v}$. Then $M=R_v^{\alpha_0}L_u^{\alpha_1}R_v^{\alpha_2}\cdots R_v^{\alpha_{r-1}}L_u^{\alpha_r}$ where $\alpha_i\in\mathbb{Z}_{\neq 0}$ for $0<i<r$, $\alpha_0,\alpha_r\in\mathbb{Z}$, and $r$ is odd. For $0
\leq i\leq r$, let $$\begin{bmatrix}
    a_i & b_i \\
    c_i & d_i
    \end{bmatrix}=\begin{cases}
    L_u^{\alpha_{r-i}}\cdots R_v^{\alpha_{r-1}}L_u^{\alpha_r} & \text{ if $i$ is even},\\
    R_v^{\alpha_{r-i}}\cdots R_v^{\alpha_{r-1}}L_u^{\alpha_r} & \text{ if $i$ is odd}.
    \end{cases}$$
Then $\begin{bmatrix}b_0\\d_0\end{bmatrix}=\begin{bmatrix}0\\1\end{bmatrix}$ and, by Lemma~\ref{cf} part (a), for $0<i\leq r$, $$\begin{bmatrix}b_i\\d_i\end{bmatrix} = \begin{cases}
\begin{bmatrix}b_{i-1}\\b_{i-1}u\alpha_{r-i}+d_{i-1}\end{bmatrix} & \text{ if $i$ is even},
\\[3ex]
\begin{bmatrix}b_{i-1}+d_{i-1}v\alpha_{r-i}\\d_{i-1}\end{bmatrix} &\text{ if $i$ is odd}.
\end{cases}$$
By repeatedly applying Lemma~\ref{cf} parts (d) and (e) to $b_i/d_i$, it follows that $b/d=b_r/d_r=[v\alpha_0,u\alpha_1,\dots,v\alpha_{r-1}].$ We must show that $(f\circ C)(b/d)=\llbracket v\alpha_0,u\alpha_1,\dots,v\alpha_{r-1}\rrbracket.$ We can assume that $r\geq 3$ since the result is trivial when $r=1.$

Since $\llbracket v\alpha_0,u\alpha_1,\dots,v\alpha_{r-1}\rrbracket\in A_2$ and $|v\alpha_{r-1}|\geq 3$, then by Lemma~\ref{scf}, $g(\llbracket v\alpha_0,u\alpha_1,\dots,v\alpha_{r-1}\rrbracket)=C(b/d)$. By Proposition~\ref{fg}, we see that
\begin{align*}
(f\circ C)\left(\frac{b}{d}\right) &= (f\circ g)(\llbracket v\alpha_0,u\alpha_1,\dots,v\alpha_{r-1}\rrbracket)\\
&= \llbracket v\alpha_0,u\alpha_1,\dots,v\alpha_{r-1}\rrbracket.
\end{align*}

$(\Leftarrow)$ We obtain the proof by following the same argument given in the first paragraph of the `if' argument in Proposition~\ref{oneFracMon}.
\end{proof}

Theorems~\ref{sanovlike} and~\ref{sanovlikeGp} show that the factorizations of matrices in $S_{u,v}$ and $G_{u,v}$ are almost completely encoded in the short continued fraction representations of the second column entries. (The single potential missing factor can be found easily enough as will be shown in an example.)   Although the encoding in $G_{u,v}$ is not as obvious as that of matrices in $S_{u,v},$ it can be deciphered in a relatively simple way. Note that Corollary~\ref{efc} shows that Theorems~\ref{sanovlike} and~\ref{sanovlikeGp} are indeed closely related to Theorems~\ref{egMon} and~\ref{egGp}.

A careful reading of Proposition~\ref{fg} shows why our hypothesis in Theorem~\ref{sanovlikeGp} requires that $u,v\geq 3.$ In order to include the cases where either $u$ or $v$ (but not both) is equal to 2 requires additional care when manipulating certain elements in $A_1$. In particular, to handle these cases we need a function like $f$ that is sensitive to both the existence \emph{and} location of partial quotients equal to 1 in elements of $A_1.$ The case $u=v=2$, Sanov's result (Theorem~\ref{sanov}), can be derived from the following lemma.

\begin{lemma}\label{22ver}
Suppose that $\alpha,\beta,\gamma\in\mathbb{Z}$ with $\beta>0$ and $\gamma\neq 0$. Then 
\begin{align*}
[\alpha,\beta,\gamma] &=[\alpha, \underbrace{1,0,1,0,\dots,0,1}_{2\beta-1\text{ terms}},\gamma]\\
&= [\alpha+1,\underbrace{-2,2,-2,\dots,2}_{\beta-1\text{ terms}},-(\gamma+1)].
\end{align*}
\end{lemma}

\begin{proof}[Proof of Theorem~\ref{sanov}]
Let $M=\begin{bmatrix}
    a & b \\
    c & d
    \end{bmatrix}\in \mathscr{G}_{2,2}$ and let $[q_0,q_1,\dots,q_r]$ be the short continued fraction representation of $b/d.$ If $r=0$, then $q_0$ must be even, so we can assume that $r>0$. By repeatedly applying Lemma~\ref{22ver} from left to right with $\alpha$ representing the leftmost odd partial quotient in the continued fraction representation of $b/d,$ we obtain a continued fraction representation $$\frac{b}{d}=[2\alpha_0,2\alpha_1,\dots,2\alpha_{s-2},\alpha_{s-1},\alpha_s].$$ By Theorem~\ref{egGp}, we must show that this continued fraction can be manipulated to an equivalent form with all even partial quotients.
    
    \noindent{\bf Case 1:} If $\alpha_{s-1}$ and $\alpha_s$ are even, the result is automatic.
    
    \noindent{\bf Case 2:} If $\alpha_{s-1}$ is odd and $\alpha_s=1$, use the fact that 
    $$[2\alpha_0,2\alpha_1,\dots,2\alpha_{s-2},\alpha_{s-1},1]=[2\alpha_0,2\alpha_1,\dots,2\alpha_{s-2},\alpha_{s-1}+1].$$
    
    \noindent{\bf Case 3:} If $\alpha_{s-1}$ is odd and $\alpha_s\neq 1$, use the fact that $$[2\alpha_0,2\alpha_1,\dots,2\alpha_{s-2},\alpha_{s-1},\alpha_s]=[2\alpha_0,2\alpha_1,\dots,2\alpha_{s-2},\alpha_{s-1},\alpha_s-1,1]$$ and apply Lemma~\ref{22ver} to the last three partial quotients.
    
    \noindent{\bf Case 4:} Suppose $\alpha_{s-1}$ is even and $\alpha_s$ is odd. If $s$ is odd, then, as in the proof of Proposition~\ref{oneFracMon}, we have a matrix $N = R_2^{-\alpha_{s-1}/2}\cdots L_2^{-\alpha_1}R_2^{-\alpha_0}M$ whose second column entries have ratio $1/\alpha_s,$ a contradiction of the fact that $N\in\mathscr{G}_{2,2}$. A similar contradiction occurs if $s$ is even. Therefore, this case cannot occur. 
\end{proof}


\section{An example}

We conclude with an example of how Theorem~\ref{sanovlikeGp} can be used.

Let $M = \begin{bmatrix}10105 & 2457 \\ -3648 & -887 \end{bmatrix}.$ Suppose that we wish to determine if $M$ is in the group $G_{4,3}.$ First note that $(10105)(-887)- (2457)(-3648)=1$, $10105\equiv 1\pmod{12}$, $2457\equiv 0\pmod{3}$, $-3648\equiv 0\pmod{4}$, and $-887\equiv 1\pmod{12}.$ So, indeed, $M\in\mathscr{G}_{4,3}.$ Now $$-\frac{2457}{887}=[-3,4,2,1,6,1,8],$$
so
\begin{align*}
    (f\circ C)\left(-\frac{2457}{887}\right) &= f(\llbracket -3,4,2,1,6,1,8\rrbracket)\\
    &= \llbracket -3,4,3\rrbracket\oplus -f(\llbracket7,1,8\rrbracket)\\
    &= \llbracket -3,4,3\rrbracket\oplus -(\llbracket 8\rrbracket\oplus - f(\llbracket9\rrbracket))\\
    &= \llbracket -3,4,3\rrbracket\oplus -(\llbracket 8\rrbracket\oplus - \llbracket9\rrbracket)\\
    &= \llbracket -3,4,3\rrbracket\oplus -\llbracket 8,-9\rrbracket\\
    &= \llbracket -3,4,3,-8,9\rrbracket.
\end{align*}
Since $\llbracket -3,4,3,-8,9\rrbracket$ satisfies the $(4,3)$-divisibility property, then $M\in G_{4,3}.$ Furthermore, $$R_3^{-3}L_4^2R_3^{-1}L_4^{-1}R_3\begin{bmatrix}10105 & 2457 \\ -3648 & -887 \end{bmatrix} = L_4,$$ so $M=R_3^{-1}L_4R_3L_4^{-2}R_3^3L_4.$


\section{Acknowledgements}
We would like to thank Alexander Rozenblyum for his careful translation of Sanov's article. The second author received support for this project provided by a PSC-CUNY award, \#61157-00 49, jointly funded by The Professional Staff Congress and The City University of New York. The authors are grateful to the anonymous referee  for the valuable comments that improved the presentation of our work.


\end{document}